\documentclass[11pt]{article}
\usepackage{amssymb,amsmath}
\usepackage[pdftex]{color}
\usepackage{amsthm}
\usepackage{amsmath}
\usepackage{amssymb}
\usepackage{graphicx}
\usepackage{amsfonts}
\usepackage{esint}
\usepackage{hyperref}

\usepackage{geometry}
\numberwithin{equation}{section}
\numberwithin{figure}{section}

\newtheorem{theorem}{Theorem}

\newtheorem{remark}{Remark}

\newcommand{\R}{\mathbb{R}}
\newcommand{\N}{\mathbb{N}}
\newcommand{\Om}{\Omega}
\newcommand{\pt}{\partial}

\DeclareMathOperator{\sign}{sign}
\numberwithin{equation}{section}
\numberwithin{theorem}{section}
\numberwithin{proposition}{section}
\numberwithin{definition}{section}
\numberwithin{corollary}{section}
\numberwithin{lemma}{section}
\numberwithin{example}{section}
\usepackage{lipsum}

\begin{document}

      \title{On some Liouville  theorems for $p$-Laplace type operators}
      
       \author{Michel Chipot \footnote{Institute of Mathematics,
           University of Z\"urich, Winterthurerstr.190, CH-8057 Z\"
           urich, email : m.m.chipot@math.uzh.ch}~{}~{}\  and Daniel
         Hauer  \footnote{Brandenburg University of Technology
           Cottbus-Senftenberg, Faculty 1 - Section Analysis, Platz der Deutschen Einheit 1,03046 Cottbus, Germany,}~{}\footnote{School of Mathematics and Statistics, The
  University of Sydney, Sydney, NSW, 2006, Australia,
email : daniel.hauer@b-tu.de, daniel.hauer@sydney.edu.au}}

 \maketitle 
 %\date{\today}
      \begin{abstract} 
  The goal of this note is to consider Liouville type theorem for
  $p$-Laplacian type operators. In particular guided by the Laplacian
  case one establishes analogous results for the $p$-Laplacian and
  operators of this type.
\end{abstract}
\vskip.5 cm
~~~~~~~~~~~~~~~~~~~~~~~~~~~~~~~~~~~~~~{\bf To Tom Sideris, an elegant scholar}
\vskip.5 cm
\noindent
{\bf MSC2020-Mathematics Subject Classification: 35A01, 35B53, 35D30, 35F25.\\[0.2cm]}


\section{Introduction and notation} 
It is well known, and it goes back to Liouville, that if $u$ is an
harmonic, bounded function in $ \R^n$ then $u$ has to be a constant,
i.e. if
\begin{displaymath}%\label{1.1}
-\Delta u = 0 ~~\text{ in }~~ {\cal D}' ({\R}^n)
\end{displaymath}
and $u$ is bounded, then $u$ is constant (see for instance \cite{E},
\cite{PW1}). The problem is much more saddle when the equation above has
a lower order term, i.e. if $u$ is a solution to the Schr\"odinger
equation
\begin{equation}\label{1.2}
-\Delta u + bu = 0 ~~\text{ in }~~ {\cal D}' ({\R}^n)
\end{equation}
for some function $b\geq 0$. If $n=2$, $b\not = 0$ then every bounded
solution to \eqref{1.2} is equal to $0$. The situation is radically
different when $n >2$. To sketch the situation, if $b$ is not decaying
too quickly at infinity, then bounded solutions to \eqref{1.2} are
vanishing. On the contrary for $b$'s with fast decay \eqref{1.2} can
have bounded non trivial solution (see \cite{BCY}, \cite{Gri1},
\cite{Gri2}, \cite{Pin2}).\medskip 

The goal of this note is to investigate the situation when the Laplacian
is replaced by the $p$-Laplacian. The expectation in this case is that
for $p \geq n$ every bounded solution to
\begin{displaymath}%\label{1.3}
-\Delta_p u = 0 ~~\text{ in }~~ {\cal D}' ({\R}^n)
\end{displaymath}
has to be constant but when $p <n$ and $b$ decays fast enough one can
exhibit nontrivial bounded solutions to
\begin{displaymath}%\label{1.4}
-\Delta_p u + b|u|^{p-2}u = 0 ~~\text{ in }~~ {\cal D}' ({\R}^n).
\end{displaymath}
This is what we would like to investigate in a slightly more general
framework. Recall that the $p$-Laplacian is defined as
\begin{displaymath}%\label{1.5}
\Delta_p u  := \pt_{x_i} \{|\nabla u|^{p-2}\pt_{x_i} u \} = \nabla \cdot \{|\nabla u|^{p-2} \nabla u\}
\end{displaymath}
with the summation convention in $i$, i.e. in the above formula one sums
in $i$ for $i=1,\cdots, n$. We will address these issues for
$p$-Laplacian type operator which archetype could be
\begin{displaymath}%\label{1.6}
 -\nabla \cdot \{a(x,u)|\nabla u|^{p-2} \nabla u\}.
\end{displaymath}
But we also discuss cases for sums of $p$-Laplace operators
\begin{displaymath}
  \pt_{x_k}\big(\sum_{i=1}^N a_i(x,u)|\nabla u |^{p_i-2}\pt_{x_k}u
\big),
\end{displaymath}
which includes the prototype operator involved in double phase problems
(see, for example, \cite{MR3775180} and references therein). 

The paper is divided as follows. The two next sections provide Liouville
type results in different situations getting in particular inspiration
from the case of the Laplacian where $b$ is chosen with a relatively
slow decay at infinity. In the Section~4 we give an example of a
nontrivial bounded solution when the lower order term of the operator
vanishes at infinity. Finally, in the last section, we briefly explain
how the arguments developped in Theorem 3.1 can be extended in the case
of several operators.

For interesting related topic one refers to 
\cite{QS}, \cite{DDH}, \cite{C5}, \cite{Mei}, \cite{PS}, \cite{Du}, \cite{PTT}.

\section{$p$-Laplacian type operators  $``p\geq n"$} 

Let us denote by $a_i(x,u)$, $i=1,\cdots, N$ Carath\'eodory functions
such that for some positive constants $\lambda, \Lambda$ one has for
$i=1,\cdots, N$
\begin{displaymath}%\label{2.1}
\lambda \leq a_i(x,u) \leq \Lambda ~~\text{ a.e. } x \in \R^n,~\forall u \in \R.
\end{displaymath}
Denote also by $b(x,u)$ a bounded Carath\'eodory function satisfying 
\begin{equation}\label{2.2}
b(x,u)u \geq 0 ~~\text{ a.e. } x \in \R^n,~\forall u \in \R.
\end{equation}
Let $p_1, \cdots, p_N$ be real numbers such that 
\begin{displaymath}%\label{2.3}
1<p_1\leq p_2 \leq \cdots \leq p_N.
\end{displaymath}
Suppose now that $u$ is a \emph{solution} to 
\begin{equation}\label{2.4}
-\pt_{x_k}\big(\sum_{i=1}^N a_i(x,u)|\nabla u |^{p_i-2}\pt_{x_k}u
\big)+ b(x,u) =0\qquad\text{in $\mathcal{D}'(\R^n)$,}
\end{equation}
i.e. $u \in W^{1,p_N}_{\ell oc}(\R^n)$ and for every bounded open
subset $\Om$ of $\R^n$
\begin{equation}\label{2.5}
\int_\Om \sum_{i=1}^N a_i(x,u)|\nabla u |^{p_i-2}\nabla u \cdot \nabla
v+ b(x,u)v =0 \qquad\forall v \in W^{1,p_N}_0(\Om).
\end{equation}

Then, one can show :
\begin{theorem} 
  \label{thm:1}
  Suppose that $p_i \geq n$, for all $i=1,\cdots, N$. Then
  the only bounded solutions to \eqref{2.4} are the constants.
\end{theorem}

\begin{proof} Set 
\begin{displaymath}%\label{2.6}
  A(x,u(x),\xi) = \sum_{i=1}^N a_i(x,u(x))|\xi |^{p_i-2}\xi
\end{displaymath}
for a.e. $x\in \Om$, and every $\xi\in \R^n$. One has, if we denote by a dot the scalar product
\begin{equation}\label{2.7}
A(x,u(x),\xi) \cdot \xi \geq \lambda \sum_{i=1}^N |\xi |^{p_i},
\end{equation}
and
\begin{equation}\label{2.8}
|A(x,u(x),\xi) | \leq \Lambda \sum_{i=1}^N |\xi |^{p_i-1}
\end{equation}
for a.e. $x\in \Om$, and every $\xi\in \R^n$. Let us denote by $\rho$ a
smooth function such that
\begin{equation}\label{2.9}
\rho = 1 \text{ on } B_{\frac{1}{2}}, ~~\rho = 0 \text{ outside } B_1, ~~|\nabla \rho | \leq K
\end{equation}
for some constant $K$ ($B_r$ denote the ball of center $0$ and radius
$r$). If $u$ is a weak solution to \eqref{2.4} and if $p \geq p_N$, then
one has that
\begin{displaymath}%\label{2.10}
v:=u\, \rho^p\left(\frac{\cdot}{r}\right) \in W^{1,p_N}_0(B_r).
\end{displaymath}
Thus from \eqref{2.5} one derives dropping the measures of integration
\begin{displaymath}%\label{2.11}
\int_{B_r} A(x,u(x), \nabla u(x))\cdot \nabla \{u \rho^p(\frac{x}{r})\}
+ b(x,u(x))u(x) \rho^p(\frac{x}{r}) = 0, 
\end{displaymath}
which is equivalent to
\begin{multline*}
 \int_{B_r} A(x,u(x), \nabla u(x))\cdot \nabla u ~\rho^p(\frac{x}{r})+ b(x,u(x))u(x) \rho^p(\frac{x}{r})\\
=- p \int_{B_r\backslash B_\frac{r}{2}}
A(x,u(x), \nabla u(x))\cdot \nabla\{ \rho (\frac{x}{r})\}  \rho^{p-1} (\frac{x}{r})u. 
\end{multline*}
Using \eqref{2.7}-\eqref{2.9}, recalling that $ \nabla\{ \rho
(\frac{x}{r})\} = \frac{1}{r} \nabla\rho (\frac{x}{r})$ we get by
\eqref{2.2} that
\begin{align*}%\label{2.12}
\lambda \int_{B_r} \sum_{i=1}^N |\nabla u|^{p_i} ~\rho^p(\frac{x}{r})&\leq \frac{pK\Lambda}{r} \sum_{i=1}^N \int_{B_r\backslash B_\frac{r}{2}}
 |\nabla u|^{p_i-1}   \rho^{p-1} (\frac{x}{r})|u|\cr 
 &\leq  \frac{pK\Lambda}{r} \sum_{i=1}^N \int_{B_r\backslash B_\frac{r}{2}}
 |\nabla u|^{p_i-1} \rho^\frac{p(p_i-1)}{p_i}  \rho^{ p -\frac{p(p_i-1)}{p_i}-1}|u|\cr
 &=  \frac{pK\Lambda}{r} \sum_{i=1}^N \int_{B_r\backslash B_\frac{r}{2}}
 |\nabla u|^{p_i-1} \rho^\frac{p}{p'_i}  \rho^{\frac{p-p_i}{p_i}}|u|\cr
\end{align*}
with $p'_i = \frac{p_i}{p_i-1}$. Using H\"older's inequality in this last integral, it comes
\begin{equation}\label{2.13}
\lambda \int_{B_r} \sum_{i=1}^N |\nabla u|^{p_i} ~\rho^p(\frac{x}{r})
\leq
 \sum_{i=1}^N \left[\int_{B_r\backslash B_\frac{r}{2}}
 |\nabla u|^{p_i} \rho^p(\frac{x}{r})\right]^{\frac{1}{p'_i}}\left[
\int_{B_r\backslash B_\frac{r}{2}}  \rho^{p-p_i}(\frac{x}{r})
|u|^{p_i}\right]^{\frac{1}{p_i}}\frac{pK\Lambda}{r}.
\end{equation}
Then, using the Young inequality 
\begin{equation}\label{2.14}
\sum_i a_ib_i \leq \varepsilon \sum_i a_i^{p'_i} + C_{\varepsilon}  \sum_i b_i^{p_i}
 \end{equation}
holding for all $\varepsilon>0$, $a_i,  b_i \geq 0$ with some constant
$C_{\varepsilon}>0$, we get
\begin{align*}%\label{2.15}
\lambda \int_{B_r} \sum_{i=1}^N |\nabla u|^{p_i} ~\rho^p(\frac{x}{r})& \leq\varepsilon \int_{B_r\backslash B_\frac{r}{2}}
\sum_{i=1}^N  |\nabla u|^{p_i} \rho^p(\frac{x}{r})+C_\varepsilon\sum_{i=1}^N \int_{B_r\backslash B_\frac{r}{2}}  \frac{1}{r^{p_i}}\rho^{p-p_i}(\frac{x}{r}) |u|^{p_i}\cr
&\leq\varepsilon \int_{B_r\backslash B_\frac{r}{2}}
\sum_{i=1}^N  |\nabla u|^{p_i} \rho^p(\frac{x}{r})+C_\varepsilon\sum_{i=1}^N \int_{B_r\backslash B_\frac{r}{2}}  \frac{1}{r^{p_i}} |u|^{p_i}.
\end{align*}
Recall that $p\geq p_i$ $\forall i$. Let us assume that 
\begin{equation}\label{2.16}
\sum_{i=1}^N  \frac{1}{r^{p_i}}\int_{B_r\backslash B_\frac{r}{2}}   |u|^{p_i} ~~\text{ is bounded independently of r. }
 \end{equation}
Then, choosing $\varepsilon= \frac{\lambda}{2}$, one derives that 
\begin{displaymath}%\label{2.17}
\int_{B_\frac{r}{2}} \sum_{i=1}^N |\nabla u|^{p_i} ~~\text{ is bounded independently of r}
\end{displaymath}
and thus, since this integral is nondecreasing in $r$ for every $i$, we
can conclude that
\begin{displaymath}%\label{2.17}
\lim_{r\to \infty} \int_{B_r} |\nabla u|^{p_i} ~~\text{exists.}
\end{displaymath}
Going back to \eqref{2.13}, applying  \eqref{2.16}, one derives easily that for some constants $C$,
\begin{align*}%\label{2.18}
\lambda \int_{B_\frac{r}{2}} \sum_{i=1}^N |\nabla u|^{p_i}& \leq C\,
\sum_{i=1}^N \left[\int_{B_r\backslash B_\frac{r}{2}}
 |\nabla u|^{p_i} \right]^{\frac{1}{p'_i}}
\left[\int_{B_r\backslash  B_\frac{r}{2}} \frac{1}{r^{p_i}} |u|^{p_i}
\right]^{\frac{1}{p_i}}\cr
 &\leq C\,\sum_{i=1}^N \left[\int_{B_r} |\nabla u|^{p_i}  - \int_{B_\frac{r}{2}}
 |\nabla u|^{p_i} \right]^{\frac{1}{p'_i}} \to 0 ~~\text{ when } r \to \infty.
\end{align*}
Thus in case \eqref{2.16} holds, $\nabla u =0$ and so, $u$ is
constant. It is easy to see that when $u$ is bounded \eqref{2.16} holds
when $p_i\geq n$ for every $i$.  This completes the proof of the
theorem.
\end{proof}

\begin{remark} Somehow the condition \eqref{2.16} is weaker than $u$
  bounded. Of course, if $b(x,u)$ is not identical equal to 0, the
  constant in Theorem 2.1 vanishes. Also using the structure assumptions
  \eqref{2.7} and \eqref{2.8}, one sees that the theorem above can be
  extended to more general operators. For instance, with a summation in $k$ for
  \begin{equation*}
  -\sum_{i=1}^N \big(\pt_{x_k} a_i^k(x,u)|\nabla u |^{p_i-2}\pt_{x_k}u
\big).   \end{equation*}
In this case the $k$-component of $A(x,u,\xi)$ is given by 
 \begin{equation*}
  -\sum_{i=1}^N a_i^k(x,u)|\xi |^{p_i-2}\xi_k   \end{equation*}
and provided $a_i^k \geq \lambda$ one has 
 \begin{equation*}
 A(x,u,\xi)\cdot\xi \geq \lambda \sum_{i=1}^N|\xi|^{p_i}  \end{equation*}
\eqref{2.8} being easy to establish if the $a_i^k$ are bounded. 

Similarly for instance for the  anisotropic Laplace operator,
 \begin{displaymath}
-\pt_{x_k} \{ a^k(x,u)|\pt_{x_k}u|^{p_k-2}\pt_{x_k} u \} 
  \end{displaymath}
(see, for example,
  \cite{MR2895946,MR2536296,MR4551772}) 
  the $k$-component of $A(x,u,\xi)$ is given by 
  \begin{equation*}
 a^k(x,u)|\xi_k|^{p_k-2}\xi_k   \end{equation*}
and provided $a^k\geq \lambda$  it holds 
 \begin{equation*}
 A(x,u,\xi)\cdot\xi = \sum_{k=1}^n   a^k(x,u)|\xi_k|^{p_k} \geq \lambda \sum_{k=1}^n|\xi_k|^{p_k}. \end{equation*}
The proof of theorem 2.1 follows the same pattern in this case, \eqref{2.13} being replaced by 
\begin{equation*}
\lambda   \sum_{k=1}^n\int_{B_r} |\partial_{x_k} u|^{p_k} ~\rho^p(\frac{x}{r})
\leq\frac{C}{r}\sum_{k=1}^n\left[\int_{B_r\backslash B_\frac{r}{2}}
 |\partial_{x_k}|^{p_k} \rho^p(\frac{x}{r})\right]^{\frac{1}{p'_k}}\left[
\int_{B_r\backslash B_\frac{r}{2}}  \rho^{p-p_k}(\frac{x}{r})
|u|^{p_k}\right]^{\frac{1}{p_k}}
\end{equation*}
and the result holds for $p_k\geq n$, $\forall k$.
\end{remark}

\section{ $p$-Laplacian type operators, ``p" arbitrary}\label{section:2} 
In this section we would like to show that, in case that the lower order
term $b(x,u)$ in equation~\eqref{2.4} is stronger, one can extend
Theorem~\ref{thm:1} to every $1<p<\infty$. To avoid technicalities we
will restrict ourselves to the case of one single operator of $p$-Laplacian
type postponing to the last section (Section~\ref{section:5}) the possible extensions. Thus for
some $p>1$, we suppose that $u$ is a \emph{solution} to 
\begin{equation}\label{3.1a}
-\pt_{x_k}\big(a(x,u)|\nabla u |^{p-2}\pt_{x_k}u\big)+ b(x,u) =0\qquad\text{in $\mathcal{D}'(\R^n)$,}
\end{equation}
i.e. $u \in W^{1,p}_{\ell oc}(\R^n)$ and for every bounded open subset $\Om$ of $\R^n$,
\begin{equation}\label{3.1b}
  \int_\Om a(x,u)|\nabla u |^{p-2}\nabla u\cdot \nabla v
  + b(x,u)v =0\qquad\forall v \in W^{1,p_N}_0(\Om).
\end{equation}
We suppose, of course, that $a(x,u)$ is a Carath\'eodory function satisfying 
\begin{equation}\label{3.2}
\lambda \leq a(x,u) \leq \Lambda ~~\text{ a.e. } x \in \R^n,~\forall u \in \R.
\end{equation}

\begin{theorem}\label{thm:2} 
 Suppose, in addition to \eqref{2.2}, that for some constant $c$ and $r$ large enough
 \begin{equation}\label{3.3}
 b(x,u)u \geq \frac{c}{r^\ell}|u|^p
\end{equation}
with $\ell <p$. Then every bounded solution to \eqref{3.1a}  vanishes.
\end{theorem}

\begin{proof} 
  Let $\rho$ be a function satisfying \eqref{2.9}. Taking as test function in \eqref{3.1b} 
  \begin{displaymath}%\label{3.4}
    v= u\, \rho^p(\frac{\cdot}{r}),
  \end{displaymath}
  we get
  \begin{displaymath}%\label{3.5}
    \int_{B_r} a(x,u)|\nabla u |^{p-2}\nabla u\cdot \nabla\{ u
    \rho^p(\frac{x}{r})\}
    + b(x,u)u \rho^p(\frac{x}{r})=0. 
\end{displaymath}
This implies easily 
\begin{equation}\label{3.6}
\int_\Om a(x,u)|\nabla u |^p\rho^p(\frac{x}{r}) + b(x,u)u \rho^p(\frac{x}{r})= - p \int_\Om a(x,u)|\nabla u |^{p-2}\nabla u \cdot \nabla \{\rho(\frac{x}{r})\} \rho^{p-1} u.
\end{equation}
Arguing as in the previous section, one derives (see \eqref{3.2}, \eqref{3.3})
\begin{equation}\label{3.7}
 \int_{B_r}\lambda |\nabla u |^{p}\rho^p +b(x,u)u \rho^p\leq \frac{pK\Lambda}{r} \int_{B_r\backslash B_\frac{r}{2}} |\nabla u |^{p-1}\rho^{p-1} |u|.
\end{equation}
Applying H\"older's inequality, it comes 
\begin{equation}\label{3.8}
\begin{aligned}
 \int_{B_r} \lambda|\nabla u |^{p}\rho^p + b(x,u)u\rho^p&\le
 \frac{pK\Lambda}{r}\, \left[\int_{B_r\backslash B_\frac{r}{2}} |\nabla u
   |^{p}\rho^{p}\right]^{\frac{1}{p'}} \left[\int_{B_r\backslash
     B_\frac{r}{2}}|u|^p\right]^{\frac{1}{p}}\cr
 &\le \frac{pK\Lambda}{r}\, \left[\int_{B_r\backslash B_\frac{r}{2}}
   |\nabla u |^{p}\rho^{p}\right]^{\frac{1}{p'}} 
 \left[\int_{B_r\backslash B_\frac{r}{2}}\frac{r^\ell}{c}b(x,u)u\right]^{\frac{1}{p}}\cr
 &\le \frac{pK\Lambda}{c^\frac{1}{p}r^{1-\frac{\ell}{p}}}\, 
 \left[\int_{B_r\backslash B_\frac{r}{2}} |\nabla u
   |^{p}\rho^{p}\right]^{\frac{1}{p'}} 
 \left[\int_{B_r\backslash B_\frac{r}{2}}b(x,u)u\right]^{\frac{1}{p}}.\cr
 \end{aligned}
\end{equation}
Using now the Young inequality 
\begin{displaymath}%\label{3.9}
ab\leq \frac{1}{p'}a^{p'} +\frac{1}{p}a^{p},~~~\forall\,a,b \geq 0,
\end{displaymath}
we get 
\begin{displaymath}%\label{3.10}
 \int_{B_\frac{r}{2}} \lambda |\nabla u |^{p} + b(x,u)u \leq \frac{pK\Lambda}{\lambda p'c^\frac{1}{p}r^{1-\frac{\ell}{p}}}\int_{B_r\backslash B_\frac{r}{2}}\lambda |\nabla u |^{p}\rho^{p} +
\frac{pK\Lambda}{ pc^\frac{1}{p}r^{1-\frac{\ell}{p}}}\int_{B_r\backslash B_\frac{r}{2}} b(x,u)u.
\end{displaymath}
Thus, for some constant $C>0$,
\begin{displaymath}%\label{3.11}
\int_{B_\frac{r}{2}} \lambda|\nabla u |^{p} + b(x,u)u\ 
\leq  \frac{C}{r^{1-\frac{\ell}{p}}}\int_{B_r} \lambda|\nabla u |^{p} + b(x,u)u.
\end{displaymath}
Iterating this formula, one derives
\begin{equation}\label{3.12}
\int_{B_\frac{r}{2^{k+1}}} \lambda |\nabla u |^{p} + b(x,u)u \leq \frac{C^k}{r^{k({1-\frac{\ell}{p}})}} \int_{B_\frac{r}{2}}  \lambda |\nabla u |^{p} + b(x,u)u.
\end{equation}
Going back to \eqref{3.8} we have 
\begin{align*}%\label{3.13}
 \int_{B_r} \lambda|\nabla u |^{p}\rho^p + b(x,u)u\rho^p
 &\le \frac{pK\Lambda}{r}\, 
 \left[\int_{B_r\backslash B_\frac{r}{2}} |\nabla u
   |^{p}\rho^{p}\right]^{\frac{1}{p'}} 
 \left[\int_{B_r\backslash B_\frac{r}{2}}|u|^p\right]^{\frac{1}{p}}\cr
 &\le \frac{pK\Lambda}{r}\frac{1}{\lambda^\frac{1}{p'}}\,\left[\int_{B_r\backslash B_\frac{r}{2}}\lambda
  |\nabla u |^{p}\rho^{p}\right]^{\frac{1}{p'}} 
 \left[\int_{B_r\backslash B_\frac{r}{2}}|u|^p\right]^{\frac{1}{p}}\cr
  &\le \frac{pK\Lambda}{r}\frac{1}{\lambda^\frac{1}{p'}}\, 
 \left[\int_{B_r\backslash B_\frac{r}{2}} \lambda|\nabla u |^{p}\rho^p +
    b(x,u)u\rho^p\right]^{\frac{1}{p'}} 
 \left[\int_{B_r\backslash B_\frac{r}{2}}|u|^p\right]^{\frac{1}{p}}
\end{align*}
and thus for some constant $C>0$,
\begin{displaymath}%\label{3.14}
  \left[\int_{B_r}  \lambda|\nabla u |^{p}\rho^p +
    b(x,u)u\rho^p\right]^{\frac{1}{p}} 
 \le \frac{C}{r}\left[\int_{B_r\backslash B_\frac{r}{2}}|u|^p\right]^{\frac{1}{p}}
\end{displaymath}
which leads  to 
\begin{displaymath}%\label{3.15}
\int_{B_\frac{r}{2}} \lambda|\nabla u |^{p} + b(x,u)u\ \le
\left(\frac{C}{r}\right)^p\,\int_{B_r\backslash B_\frac{r}{2}}|u|^p.
\end{displaymath}
If $u$ is supposed to be uniformly bounded, then one gets 
\begin{equation}\label{3.16}
\int_{B_\frac{r}{2}} \lambda|\nabla u |^{p} + b(x,u)u\leq Cr^{n-p}.
\end{equation}
for some other constant $C$. From \eqref{3.12}, we derive then 
\begin{displaymath}%\label{3.17}
\int_{B_\frac{r}{2^{k+1}}} \lambda |\nabla u |^{p} + b(x,u)u \le 
\frac{C}{r^{k({1-\frac{\ell}{p}})}}r^{n-p} \to 0  \text{ when }r\to \infty  \text{ and  } k({1-\frac{\ell}{p}}) >p-n.
\end{displaymath}
This completes the proof of Theorem~\ref{thm:2}.
\end{proof}

\begin{remark} From \eqref{3.16} one can get the result for $p >n$. Note also that \eqref{3.3} holds with $\ell=0$ when one has 
 \begin{equation}\label{3.18}
 b(x,u)u \geq c |u|^p \qquad\text{for a.e. $x\in\R^d$ and every $u\in \R$.}
\end{equation}
\end{remark}

\section{Existence of a nontrivial solution}\label{section:4}

In this section, we would like to construct a nontrivial bounded
solution to the equation
\begin{equation}
  \label{4.0a}
  -\Delta_p u + b\,u=0\qquad\text{in $\mathcal{D}'(\R^n)$,}
\end{equation}
when $b=b(x)$ is nonnegative. Here, a function $u$ is call a
\emph{solution} to~\eqref{4.0a} if $u \in W^{1,p}_{\ell oc}(\R^n)$ and
for every open bounded subset $\Om\subseteq \R^n$,
\begin{equation}\label{4.0}
\int_\Om |\nabla u |^{p-2}\nabla u \cdot \nabla v+ b(x)|u|^{p-2}uv =0 ~~\forall v \in W^{1,p}_0(\Om). 
\end{equation}

Recall that $B_k$ denotes the ball of center $0$ and radius $k$. Then,
for every $k\in \N$, there exists a unique solution $u_k$ to the variational inequality 
 \begin{equation}\label{4.1}
\begin{cases}
u_k \in K = \{ v \in W^{1,p}(B_k) : v=1 \text{ on } \pt B_k \}, \cr
\cr
\displaystyle
\int_{B_k}  |\nabla u_k |^{p-2}\nabla u_k \cdot \nabla (v -u_k) + b(x)|u_k |^{p-2}u_k(v -u_k)  \geq 0 ~~\forall v \in K. \cr
\end{cases}
\end{equation}
We refer, for instance, to \cite{KS}, \cite{C4} or to the
Remark~\ref{rem:3} below.

\vskip .5 cm 1. Claim: $0\le u_k \le 1$ on $B_k$ \vskip .5 cm Recall that $w^+(x):=\max\{0,w(x)\}$ denotes the
positive part of a function $w$ and $w^{-}:=(-w)^+$ the negative
part. Then, taking $v= u_k^+$ as a test function in \eqref{4.1} and by using
that $u_k^+-u_k=u_k^-$, it comes
 \begin{displaymath}%\label{4.2}
\int_{B_k}  |\nabla u_k |^{p-2}\nabla u_k \cdot \nabla u_k^-+ b|u_k |^{p-2}u_ku_k^-  = -\int_{B_k}  |\nabla u_k^- |^{p-2}\nabla u_k^-\cdot \nabla u_k^-+ b |u_k |^{p-2}u_k^-u_k^- \ge 0, 
\end{displaymath}
from where we can conclude that
 \begin{displaymath}%\label{4.3}
\int_{B_k}  |\nabla u_k^- |^{p}+ b |u_k^-|^p \le 0. 
\end{displaymath}
Thus, $u_k^- =0$ on $B_k$, which implies that $u_k\ge 0$ on $B_k$.
\vskip .5 cm
Taking $v = u_k \pm (u_k-1)^+$ in \eqref{4.1}, one gets 
 \begin{displaymath}%\label{4.4}
\int_{B_k}  |\nabla u_k |^{p-2}\nabla u_k \cdot \nabla (u_k -1)^+ + b|u_k |^{p-2}u_k (u_k-1)^+  =  0 
\end{displaymath}
and hence, 
 \begin{displaymath}%\label{4.5}
\int_{B_k}  |\nabla u_k |^{p-2}\nabla ( u_k -1) \cdot \nabla (u_k -1)^+ = - \int_{B_k}  b|u_k |^{p-2}u_k (u_k-1)^+  \le 0 . 
\end{displaymath}
Thus $(u_k-1)^+ = 0$, i.e. $u_k\le 1$.
\vskip .5 cm

2. Claim: $u_{k+1} \leq u_k$ on $B_k$
\vskip .5 cm
Clearly $(u_{k+1}  -u_k)^+ \in W^{1,p}_0(B_k)$. We now suppose that this
function is extended by $0$ on $B_{k+1}$. Taking $v= u_k \pm (u_{k+1}
-u_k)^+$ in \eqref{4.1}, we get that 
 \begin{displaymath}%\label{4.6}
\int_{B_k}  |\nabla u_k |^{p-2}\nabla u_k   \cdot \nabla(u_{k+1}  -u_k)^+ +   b |u_k |^{p-2}u_k (u_{k+1}  -u_k)^+  =  0   . 
\end{displaymath}
Similarly, taking $v= u_{k+1} \pm (u_{k+1}  -u_k)^+$ in \eqref{4.1} gives
 \begin{displaymath}%\label{4.7}
\int_{B_k}  |\nabla u_{k+1}|^{p-2}\nabla u_{k+1}   \cdot \nabla(u_{k+1}
-u_k)^+ +   b | u_{k+1}|^{p-2} u_{k+1} (u_{k+1}  -u_k)^+  =  0   . 
\end{displaymath}
By subtraction, it comes 
 \begin{displaymath}%\label{4.8}
 \begin{aligned}
\int_{B_k} \{ |\nabla u_{k+1}|^{p-2}\nabla u_{k+1} -  &|\nabla u_k |^{p-2}\nabla u_k\} \cdot \nabla(u_{k+1}  -u_k)^+ \cr
&+   b\{ | u_{k+1}|^{p-2} u_{k+1} -  |u_k |^{p-2}u_k\}(u_{k+1}  -u_k)^+  =  0   . 
 \end{aligned}
 \end{displaymath}
Thus for some constant $c_p>0$, it comes (see, for example, \cite{C4})
 \begin{displaymath}%\label{4.9}
c_p \int_{B_k} \big( |\nabla u_{k+1}| +|\nabla u_k| \big)^{p-2}| \nabla(u_{k+1}  -u_k)^+ |^2\le  0, 
\end{displaymath}
implying that $(u_{k+1}  -u_k)^+ =0$ on $B_k$, which is $u_{k+1}  \le
u_k$ on $B_k$.
\vskip .5 cm

From the Claim 1. and Claim 2., we derive that
\begin{equation}
  \label{eq:pointwise-limit}
  u_k(x) \to u(x)\qquad\text{pointwise for a.e. $x\in \R^n$,} 
\end{equation}
 where $u : \R^n\to \R$ is a function satisfying
 \begin{displaymath}%\label{4.10}
   0\le u\le 1\qquad\text{on $\R^n$}.   
\end{displaymath}
\vskip .5 cm

3. Claim: If $b$ is radially symmetric, so is $u_k$ and $u$.
\vskip .5 cm
If $R=(R_{j,k})$ is an orthogonal transformation, then one has with the summation convention
 \begin{displaymath}%\label{4.11}
\begin{aligned}
\nabla \{v(Rx)\} &= (\pt_{y_j} v(Rx)  \pt_{x_i} R_{j,k}x_k) \cr
&=(R_{j,i}\pt_{y_j} v(Rx)) = R^T\{\nabla v\} (Rx).
 \end{aligned}
\end{displaymath}
Thus one has, by a change of variable 
 \begin{displaymath}%\label{4.12}
\begin{aligned}
\int_{B_k}  |\nabla \{u_k(Rx)\} |^{p-2}\nabla \{u_k (Rx)\}& \cdot \nabla \{(v(Rx) -u_k(Rx))\} \cr
&+ b
|u_k (Rx)|^{p-2}u_k(Rx)(v (Rx)-u_k(Rx))  \geq 0  \cr
 \end{aligned}
\end{displaymath}
for any $v \in W^{1,p}_0(B_k), ~v=1 $ on $\pt B_k$. Choosing $v(R^Tx)$ we see, by uniqueness of $u_k$  that 
 \begin{displaymath}%\label{4.13}
u_k(Rx) = u_k(x)
\end{displaymath}
for any orthogonal transformation $R$.

\begin{remark}\label{rem:3} 
 Taking $v=u_k\pm \varphi$ for $\varphi \in W^{1,p}_0(B_k)$ in \eqref{4.1}, one sees that $u_k$ satisfies 
 \begin{equation}\label{4.14}
u_k\in K\quad\text{ and }\quad
\int_{B_k}  |\nabla u_k|^{p-2}\nabla u_k \cdot \nabla \varphi  + b| u_k|^{p-2} u_k\varphi = 0 ~~\forall \varphi \in W^{1,p}_0(B_k),
\end{equation}
that is, $u_k$ is a \emph{weak solution} of the nonlinear Dirichlet problem
\begin{displaymath}
\begin{alignedat}{2}
	   -\Delta_{p}u_k+b\,|u_k|^{p-2}u_k&=0\quad && 
	\text{in $B_k$,}\\
	u_k&=1\quad && \text{on $\partial B_k$.}
       \end{alignedat}
\end{displaymath}
Note that $u_k$ is also the unique minimiser on $K$ to
 \begin{displaymath}%\label{4.15}
   J(v) = \int_{B_k}  |\nabla v|^{p} + b|v|^p. 
\end{displaymath}
\end{remark}

From now on, we suppose that 
 \begin{equation}\label{4.16}
 \begin{aligned}
&b \text{ is radially symmetric with compact support, i.e. } \cr
&~~~~~~~~~~~~~~b(x)=b(|x|) =0\qquad \text{ for all} ~|x|=r \ge r_0.\cr
\end{aligned}
\end{equation}
Since the function $1\in K$, one has then 
\begin{displaymath}%\label{4.17}
 \int_{B_k}  |\nabla u_k|^{p} + b|u_k|^p=J(u_k) \le J(1) = \int_{\R^n} b < +\infty.
 \end{displaymath}
Thus, up to a subsequence, 
\begin{equation}\label{4.18}
 \nabla u_k \rightharpoonup  \nabla u \text{ in } L^p(\Om)
 \end{equation}
for every bounded subdomain $\Om$ of $\R^n$.
\vskip .5 cm

4. Differential equation satisfied by $u_k$ and $u$.
\vskip .5 cm
If $u_k=u_k(r)$, then 
\begin{displaymath}%\label{4.19}
 \nabla u_k = u'_k(r)\nabla r = u'_k(r)\frac{x}{r}\quad\text{ and
 }\quad | \nabla u_k| = |u'_k(r)|.
 \end{displaymath}
From this, it follows that 
 \begin{displaymath}%\label{4.20}
 \begin{aligned}
\nabla \cdot (| \nabla u_k |^{p-2} \nabla u_k ) & =  \pt_{x_i}(|u'_k|^{p-2} u'_k \frac{x_i}{r})\cr
&=|u'_k|^{p-2} u'_k \pt_{x_i}\{\frac{x_i}{r}\} + (|u'_k|^{p-2} u'_k )'\frac{x_i}{r}\frac{x_i}{r}\cr
&=|u'_k|^{p-2} u'_k (\frac{n}{r}) + |u'_k|^{p-2} u'_k  x_i \big(-\frac{1}{r^2}\big)\frac{x_i}{r}\ + (|u'_k|^{p-2} u'_k )'\cr
&= |u'_k|^{p-2} u'_k (\frac{n-1}{r}) +(|u'_k|^{p-2} u'_k )'\cr
&= \frac{1}{r^{n-1}}\Big(  |u'_k|^{p-2} u'_k (n-1) r^{n-2} + r^{n-1}(|u'_k|^{p-2} u'_k )'\Big)\cr
&= \frac{1}{r^{n-1}} (|u'_k|^{p-2} u'_k  r^{n-1})' .\cr
 \end{aligned}
\end{displaymath}
Thus from \eqref{4.14}, one derives that $u_k$ satisfies
\begin{align*}
  \frac{1}{r^{n-1}} (|u'_k|^{p-2} u'_k  r^{n-1})' &= b|u_k|^{p-2}
  u_k\qquad\text{for $0<r<k$,}\\
\intertext{which is equivalent to}
 (|u'_k|^{p-2} u'_k  r^{n-1})' &= r^{n-1}b|u_k|^{p-2} u_k\qquad\text{for $0<r<k$,}\\
\intertext{and again, equivalent to}
|u'_k(r)|^{p-2} u'_k(r)   &=\frac{1}{r^{n-1}}\int_0^r s^{n-1}b |u_k|^{p-2} u_k ds\qquad\text{for $0<r<k$.}
\end{align*}
Setting $\Psi(x) = |x|^{p-2} x$ for $x\in \R$, then $\psi$ is bijective on $\R$ and
its inverse is $\Psi^{-1}(x) = |x|^{\frac{1}{p-1}}\sign x$, where $\sign x$
denotes the sign of $x$. One gets 
\begin{equation}\label{4.23}
u'_k= \Psi^{-1}(\frac{1}{r^{n-1}}\int_0^r s^{n-1}b |u_k|^{p-2} u_k ds)\qquad\text{for $0<r<k$.}
 \end{equation}
From \eqref{4.18}, one has up to a subsequence still labelled by $k$
\begin{equation}\label{4.24}
\nabla u_k = u'_k \frac{x}{r}\rightharpoonup u' \frac{x}{r} \text{ in } L^p(\Om)
 \end{equation}
 for every open and bounded subset $\Om\subseteq \R^n$. Thus and by
 using~\eqref{eq:pointwise-limit}, multiplying~\eqref{4.23} with $x/r$
 for $r>0$ and subsequently passing to the limit, we arrive
 to
\begin{align*}%\label{4.25}
u'(r)\frac{x}{r}&= \frac{x}{r}\Psi^{-1}(\frac{1}{r^{n-1}}\int_0^r
                  s^{n-1}b |u|^{p-2} u ds)\qquad\text{for $r>0$,}\\
\intertext{which is equivalent to}
u'(r)&= \Psi^{-1}(\frac{1}{r^{n-1}}\int_0^r s^{n-1}b |u|^{p-2} u
       ds)\qquad\text{for $r>0$,}\\
\intertext{and}
\Psi(u') = |u'|^{p-2} u' &= \frac{1}{r^{n-1}}\int_0^r s^{n-1}b |u|^{p-2} u ds\qquad\text{for $r>0$.}
 \end{align*}
Multiplying the last equation by $r^{n-1}$ and subsequently,
differentiating it, shows that $u$ satisfies
\begin{displaymath}%\label{4.26}
 - \frac{1}{r^{n-1}} (|u'|^{p-2} u'  r^{n-1})' + b|u|^{p-2} u=0\qquad\text{in $(0,\infty)$,}
\end{displaymath}
that is, $u$ satisfies the same equation as $u_k$ in all $\R^n$.\medskip

We would like to show now that $u$ is nontrivial.

\vskip .5 cm
5. The limit of $u_k$ cannot be identically $0$, that is, $u$ is nontrivial.
\vskip .5 cm

Due to the definition of $b$, one has that 
\begin{displaymath}%\label{4.27}
(|u'_k|^{p-2} u'_k  r^{n-1})'=0 \text{ for } r \geq r_0.
 \end{displaymath}
Thus 
\begin{displaymath}%\label{4.28}
|u'_k|^{p-2} u'_k  r^{n-1}= C_k \text{ for } r \geq r_0.
 \end{displaymath}
where $C_k$ is some constant. Thus for $r \geq r_0$ one has 
\begin{displaymath}%\label{4.29}
u'_k  = \Psi^{-1}(\frac{C_k }{r^{n-1}}) = |C_k|^\frac{1}{p-1} \sign C_k \frac{1}{r^\frac{n-1}{p-1}} .
 \end{displaymath}
Integrating between $r_0$ and $r$, we get 
\begin{equation}\label{4.30}
u_k(r) - u_k(r_0)   =  |C_k|^\frac{1}{p-1} \sign C_k\,\int_{r_0}^r\frac{1}{r^\frac{n-1}{p-1}}.
 \end{equation}

 Now, if $u_k(r) \to 0$ pointwise, \eqref{4.30}
 implies that $C_k\to 0$. On the other hand, choosing $r=k$ in \eqref{4.30} gives that
\begin{equation}\label{4.31}
1- u_k(r_0)   =  |C_k|^\frac{1}{p-1} \sign C_k\,\int_{r_0}^k\frac{1}{r^\frac{n-1}{p-1}}
 \end{equation}
for every $k\geq r_0$. If $n>p$, then the integral above converges 
 and so,  we arrive to a contradiction when we send $k \to
\infty$ in \eqref{4.31}. Thus we have proved

\begin{theorem}\label{thm:3} 
  In the case $n>p$ ($n>2$ in the case of the Laplacian), one can find
  $b$ satisfying \eqref{4.16} such that \eqref{4.0a} admits a nontrivial
  bounded solution.
\end{theorem}
\vskip .5 cm
\section{Concluding remark}\label{section:5}

We would like to show briefly here how Theorem~\ref{thm:2} can be
extended in the case of several $p$-Laplacian type operators. Suppose
that $u$ is a solution to \eqref{2.4}.  Arguing as in \eqref{3.6} and
\eqref{3.7},  one gets that
\begin{equation}\label{5.1}
 \int_{B_r}\lambda \sum_{i=1}^N |\nabla u|^{p_i} ~\rho^p(\frac{x}{r}) + b(x,u)u\rho^p(\frac{x}{r})\leq  \frac{pK\Lambda}{r}  \int_{B_r\backslash B_\frac{r}{2}}
 \sum_{i=1}^N|\nabla u|^{p_i-1} \rho^\frac{p}{p'_i}  \rho^{\frac{p-p_i}{p_i}}|u|.
\end{equation}
Using the H\"older  inequality we derive 
\begin{displaymath}%\label{5.2}
\begin{aligned}
 \int_{B_r}\lambda \sum_{i=1}^N |\nabla u|^{p_i} ~\rho^p(\frac{x}{r}) &+ b(x,u)u\rho^p(\frac{x}{r})\cr 
 &\leq  \frac{pK\Lambda}{r}  \sum_{i=1}^N \Big( \int_{B_r\backslash B_\frac{r}{2}}
 |\nabla u|^{p_i} \rho^p\Big)^\frac{1}{p'_i}\Big(   \int_{B_r\backslash B_\frac{r}{2}} \rho^{p-p_i}|u|^{p_i}\Big)^\frac{1}{p_i}.\cr
\end{aligned}
\end{displaymath}
Assuming then for $x$ large enough and for all $i$
\begin{displaymath}%\label{5.3}
 b(x,u)u \geq \frac{c}{r^\ell}|u|^{p_i},~~c>0, \ell <p_1 \leq p_i
 \end{displaymath}
we get 
\begin{displaymath}%\label{5.4}
\begin{aligned}
 \int_{B_r}\lambda \sum_{i=1}^N |\nabla u|^{p_i} ~\rho^p(\frac{x}{r}) &+ b(x,u)u\rho^p(\frac{x}{r})\cr 
 &\leq  \frac{pK\Lambda}{r^{1-\frac{\ell}{p_1}}}  \sum_{i=1}^N \Big( \int_{B_r\backslash B_\frac{r}{2}}
 |\nabla u|^{p_i} \rho^p\Big)^\frac{1}{p'_i}\Big(   \int_{B_r\backslash B_\frac{r}{2}} \rho^{p-p_i}\frac{1}{c} b(x,u)u \Big)^\frac{1}{p_i}.\cr
\end{aligned}
\end{displaymath}
Then applying the Young inequality 
\begin{displaymath}%\label{5.5}
ab \leq \frac{1}{p'_i} a^{p'_i} + \frac{1}{p_i} b^{p_i},~~a, b \geq 0 
\end{displaymath}
it comes easily for some constant $C$
\begin{displaymath}%\label{5.6}
 \int_{B_r}\lambda \sum_{i=1}^N |\nabla u|^{p_i} ~\rho^p(\frac{x}{r}) + b(x,u)u\rho^p(\frac{x}{r})\leq \frac{C}{r^{1-\frac{\ell}{p_1}}} \int_{B_r\backslash B_\frac{r}{2}}
 \sum_{i=1}^N \lambda|\nabla u|^{p_i} \rho^{p} +  \rho^{p-p_i}b(x,u)u
\end{displaymath}
and thus, if $p\geq p_i $, for some constant $C$, we get
\begin{displaymath}%\label{5.7}
 \int_{B_\frac{r}{2}}\lambda \sum_{i=1}^N |\nabla u|^{p_i}  + b(x,u)u\leq  \frac{C}{r^{1-\frac{\ell}{p_1}}}  \int_{B_r}
 \lambda  \sum_{i=1}^N |\nabla u|^{p_i}  +b(x,u)u.
\end{displaymath}

Iterating this formula, one gets
\begin{equation}\label{5.8}
 \int_{B_\frac{r}{2^{k+1}}}\lambda \sum_{i=1}^N |\nabla u|^{p_i}  + b(x,u)u\leq \frac{C^k}{r^{k(1-\frac{\ell}{p_1})}}   \int_{B_\frac{r}{2}}
 \lambda \sum_{i=1}^N  |\nabla u|^{p_i}  + b(x.u)u.
\end{equation}
Going back to \eqref{5.1} and using \eqref{2.14}  (taking $\varepsilon =
\frac{1}{2}$), we obtain that 
\begin{displaymath}%\label{5.9}
 \int_{B_r}\lambda \sum_{i=1}^N |\nabla u|^{p_i}\rho^p  + b(x,u)u\rho^p \leq  
 \varepsilon  \int_{B_r}\lambda \sum_{i=1}^N |\nabla u|^{p_i}\rho^p +
 C_\varepsilon  \int_{B_r} \sum_{i=1}^N  \frac{|u|^{p_i}}{r^{p_i}}
 \end{displaymath}
and
\begin{displaymath}%\label{5.10}
 \int_{B_\frac{r}{2}}\lambda \sum_{i=1}^N |\nabla u|^{p_i}\  + b(x,u)u \leq  
2 \,C_\varepsilon  \int_{B_r} \sum_{i=1}^N  \frac{|u|^{p_i}}{r^{p_i}}.
 \end{displaymath}
If $u$ is bounded, this leads to 
\begin{displaymath}%\label{5.11}
 \int_{B_\frac{r}{2}}\lambda \sum_{i=1}^N |\nabla u|^{p_i}\  + b(x,u)u \leq  
C  \sum_{i=1}^N r^{n-p_i}.
 \end{displaymath}
By \eqref{5.8}, it follows that
\begin{displaymath}%\label{5.12}
 \int_{B_\frac{r}{2^{k+1}}}\lambda \sum_{i=1}^N |\nabla u|^{p_i}  + b(x,u)u
 \leq  C \sum_{i=1}^N \frac{1}{r^{k(1-\frac{\ell}{p_1})+p_i-n}} \to 0 \text{ for } k(1-\frac{\ell}{p_1})> n-p_i.
\end{displaymath}
This completes the proof in this case.

\vskip .5 cm
{\bf Acknowledgements} {The second author's research was partially supported by the two
  Australian Research Council grants DP200101065 and DP220100067.}

\vskip .5 cm

%\bibliography{Ref}

\begin{thebibliography}{10}

\bibitem{MR3775180}
P.~Baroni, M.~Colombo, and G.~Mingione.
\newblock Regularity for general functionals with double phase.
\newblock {\em Calc. Var. Partial Differential Equations}, 57(2):No. 62,
  48, 2018.

\bibitem{MR4551772}
B.~Brandolini and F.~C. C\^{i}rstea.
\newblock Anisotropic elliptic equations with gradient-dependent lower order
  terms and {$L^1$} data.
\newblock {\em Math. Eng.}, 5(4):No. 073, 33, 2023.

\bibitem{BCY}
H.~Brezis, M.~Chipot, and Y.~Xie.
\newblock Some remarks on {L}iouville type theorems.
\newblock In World~Scientific Edt., editor, {\em proceedings of the
  international conference in nonlinear analysis}, 43--65, Hsinchu,
  Taiwan 2006, 2008.

\bibitem{C4}
M.~Chipot.
\newblock {\em Elliptic Equations: An Introductory Course}.
\newblock Birkh\"auser, 2009.

\bibitem{C5}
M.~Chipot.
\newblock {\em Elliptic Equations: An Introductory Course, Second edition}.
\newblock Birkh\"auser, 2024.

\bibitem{DDH}
E.N. Dancer~D. Daners and D.~Hauer.
\newblock A {L}iouville theorem for p-harmonic functions on exterior domains.
\newblock {\em Positivity}, 19:577--586, 2015.

\bibitem{Du}
Y.~Du.
\newblock {\em Order Structure and Topological Methods in Nonlinear Partial
  Differential Equations}.
\newblock World Scientific, Taipei, 2006.

\bibitem{E}
L.~C. Evans.
\newblock {\em Partial Differential Equations}, Volume~19 of {\em Graduate
  Studies in Mathematics}.
\newblock American Math.\ Society, Providence, 1998.

\bibitem{Gri1}
A.~Grigor'yan.
\newblock Bounded solutions of the {S}chr\"odinger equation on non-compact
  {R}iemannian manifolds.
\newblock {\em J. Sov. Math.}, 51:2340--2349, 1990.

\bibitem{Gri2}
A.~Grigor'yan and W.~Hansen.
\newblock A {L}iouville property for {S}chr\"odinger operators.
\newblock {\em Math. Ann.}, 312:659--716, 1998.

\bibitem{KS}
D.~Kinderlehrer and G.~Stampacchia.
\newblock {\em An Introduction to Variational Inequalities and their
  Applications}, volume~31 of {\em Classic Appl. Math.}
\newblock SIAM, Philadelphia, 2000.

\bibitem{Mei}
M.~Meier.
\newblock {L}iouville theorem for nonlinear elliptic equations and systems.
\newblock {\em Manuscripta Mathematica}, 29:207--228, 1979.

\bibitem{Pin2}
R.~G. Pinsky.
\newblock A probabilistic approach to a {L}iouville-type problem for
  {S}chr\"odinger operators.
\newblock Preprint, 2006.

\bibitem{PW1}
M.~H. Protter and H.~F. Weinberger.
\newblock {\em Maximum Principles in Differential Equations}.
\newblock Prentice-Hall, Englewood Cliffs, NJ, 1967.

\bibitem{PS}
P.~Pucci and J.~Serrin.
\newblock {\em The Maximum Principle}, volume \#73 of {\em Progress in
  Nonlinear Differential Equations and Their Applications}.
\newblock Birkh\"auser, 2007.

\bibitem{QS}
P.~Quittner and P.~Souplet.
\newblock {\em Superlinear Parabolic Problems, Blow-up, Global Existence and
  Steady States}.
\newblock Birkh\"auser, 2007.

\bibitem{PTT}
Y.~Pinchover~A. Tertikas and K.~Tintarev.
\newblock A {L}iouville-type theorem for the p-laplacian with potential terms.
\newblock {\em Ann. I. H. Poincar\'e}, 25:357--368, 2008.

\bibitem{MR2536296}
J. V\'{e}tois.
\newblock A priori estimates for solutions of anisotropic elliptic equations.
\newblock {\em Nonlinear Anal.}, 71(9):3881--3905, 2009.

\bibitem{MR2895946}
J. V\'{e}tois.
\newblock Strong maximum principles for anisotropic elliptic and parabolic
  equations.
\newblock {\em Adv. Nonlinear Stud.}, 12(1):101--114, 2012.

\end{thebibliography}
%\bibliographystyle{plain}

\end{document}